\documentclass[11pt]{amsart}
\usepackage{amsmath,amssymb,amsthm}
\usepackage{tikz-cd} 
\usepackage[margin=1in]{geometry}
\usepackage{amssymb,amsthm}
\usepackage{mathtools}
\usepackage{eucal,stmaryrd}
\usepackage{enumerate}
\usepackage[margin=1in]{geometry}
\usepackage{quiver}
\usepackage{hyperref}
\usepackage{xcolor}
 \usepackage{enumitem}

\hypersetup{
    colorlinks,
    linkcolor={red!50!black},
    citecolor={blue!50!black},
    urlcolor={blue!80!black}
}
\theoremstyle{plain}
\newtheorem{theorem}{Theorem}[section]
\newtheorem{proposition}[theorem]{Proposition}
\newtheorem{lemma}[theorem]{Lemma}

\theoremstyle{definition}

\newtheorem{remark}[theorem]{Remark}

\newcommand{\tp}{{\scriptscriptstyle\mathsf{T}}}

\newcommand{\FS}{{\scriptscriptstyle{FS}}}
\newcommand{\C}{{\scriptscriptstyle{C}}}

\DeclareMathOperator{\rank}{rank}
\DeclareMathOperator{\spa}{span}
\DeclareMathOperator{\Gr}{Gr}
\DeclareMathOperator{\tr}{tr}

\begin{document}
\title{Upper Bounds for $s$-distance Subspaces}
\author{Lixia Wang}
\address{State Key Laboratory of Mathematical Sciences, Academy of Mathematics and Systems Science, Chinese Academy of Sciences}
\email{wanglixia@amss.ac.cn}
\author{Ke Ye}
\address{State Key Laboratory of Mathematical Sciences, Academy of Mathematics and Systems Science, Chinese Academy of Sciences}
\email{keyk@amss.ac.cn}

\begin{abstract}
As a generalization of equiangular lines, equiangular subspaces were first systematically studied by Balla, Dr\"{a}xler,  Keevash and Sudakov in 2017.  In this paper,  we extend their work to $s$-distance subspaces,  i.e.,  to sets of $k$-dimensional subspaces in $\mathbb{R}^n$ whose pairwise distances take $s$ distinct values.  We establish upper bounds on the maximum cardinality of such sets.  In particular,  our bounds generalize and improve results of Balla and Sudakov.
\end{abstract}

\maketitle

\section{Introduction}
The study of $s$-distance subsets---finite subsets of a metric space whose pairwise distances take $s$ distinct non-zero values---dates back to the 1960s \cite{ES66a,ES66b},  although the special case of equilateral point sets \cite{Haantjes48} was investigated much earlier.  Given a metric space,  determining the maximum cardinality of an $s$-distance set is a central problem.  Typical examples include equiangular lines \cite{BDKS18,JTYZZ21} and subspaces \cite{LS73,EF09,Balla} in $\mathbb{R}^n$; equilateral points in $\ell_p^n$ \cite{AP03} and in elliptic spaces \cite{Haantjes48,vS66}; and $s$-distance sets in the unit sphere \cite{DGS77,Bukh16,GY18}. 

Because of its great importance in graph theory \cite{GR01},  discrete geometry \cite{LS73} and coding theory \cite{Delsarte73},  the study of equiangular lines has remained  active in the past few decades \cite{LS73,Neumaier89,Bukh16,BDKS18,JTYZZ21}. To generalize equiangular lines, Lemmens and Seidel introduced equi-isoclinic subspaces \cite{LS73,EF09,Et-Taoui18,FIJM24}.  In a different direction,  Blokhuis \cite{Blokhuis93} initiated the investigation of equiangular planes in 1993. Only very recently,  Balla, Dr\"{a}xler,  Keevash and Sudakov extended this line of work to general equiangular subspaces \cite{IFPB17}.  

The Grassmannian $\Gr(k,n)$ is the set of all $k$-dimensional linear subspaces of $\mathbb{R}^n$.  Endowed with various distance functions, it has been studied extensively as a metric space for over ninety years \cite{Friedrichs37,Dixmier49,Wong67,Asimov85,
CHS96,BN02,DHST08,DD14}.  Suppose $\delta: \Gr(k,n) \times \Gr(k,n) \to \mathbb{R}$ is a function.  A finite subset $S \subseteq \Gr(k,n)$ is called an \emph{$s$-distance set} with respect to $\delta$ if 
\[
\left\lvert \{ \delta(\mathbb{U}, \mathbb{V}): \mathbb{U} \neq \mathbb{V},  \; \mathbb{U}, \mathbb{V} \in S \} \right\rvert \le s.
\]
By definition,  a set of equiangular subspaces is a $1$-distance subset of $\Gr(k,n)$.  We denote 
\[
N^{\delta}_s(k,n) \coloneqq \max \left\lbrace
|S|: S \subseteq \Gr(k,n) \text{~is an $s$-distance set with respect to $\delta$}
\right\rbrace.
\] 
It was proved in \cite{CHS96} that 
\begin{equation}\label{eq:NChordal}
N_{1}^{d_{\C}}(k,n) \le \binom{n+1}{2}.
\end{equation} 
Here $d_{\C}$ denotes the Chordal distance (see~\eqref{eq:Chordal}).  Moreover,  \cite{IFPB17,Balla} showed that for the Fubini-Study distance $d_{\FS}$ (see~\eqref{eq:FS}),
\begin{equation}\label{eq:NdFS}
N_1^{d_{\FS}}(k,n) \le \binom{\binom{n}{k} + 1}{2} = \frac{1}{2 (k!)^2} n^{2k} + O(n^{2k-1}).
\end{equation}
In the same work,  an upper bound for $N_1^{\delta}(k,n)$ was also obtained for an angle distance $\delta$ (see Section~\ref{sec:equiangular}):
\begin{equation}\label{eq:Ndelta}
N_1^{\delta}(k,n) \le \binom{\binom{n+1}{2} + k -1}{k}.
\end{equation} 

This paper is devoted to establishing upper bounds on $N_s^{\delta}(k,n)$ for general $s$.  Our main goal is to generalize and improve \eqref{eq:NChordal}--\eqref{eq:Ndelta} by exploiting the geometry of $\Gr(k,n)$.  In fact,  Balla and Sudakov mentioned in \cite[Page~88]{Balla} that it is conceivable to obtain a better upper bound for $N_1^{d_{\FS}}(k,n)$ by investigating the geometry of $\Gr(k,n)$.  The results of this paper confirm and extend this expectation. 

\subsection*{Main results}
In Theorem~\ref{thm:s-dC} and Proposition~\ref{cor:spherical},  we prove that 
\[
N_{s}^{d_{\C}}(k,n) \le \begin{cases}
\frac{d(k,n)}{(k(n-k))!}  s^{k(n-k)} + O(s^{k(n-k)-1}) \quad &\text{if $s$ is sufficiently large}, \\[1em]
\binom{n + 2s - 1}{2s} + \binom{n + 2s -2}{2s-1} \quad &\text{if $k = 1$}.
\end{cases}
\] 
Here $d(k,n)$ is an explicitly computable constant depending on $k$ and $n$.  Our bound for $N_1^{d_{\C}}(1,n)$ is of the same order as \eqref{eq:NChordal}.  

For $N_s^{d_\FS}(k,n)$,  we establish in Theorem~\ref{thm:s-distFS} the following upper bound:
\[
N_s^{d_\FS}(k,n) \le \left[\prod_{j=1}^{2s} \frac{(j-1)!}{(j+k-1)!}\right] n^{2ks}  + O\left(n^{2ks-1}\right),
\]
where $k \le n$ and $s$ are arbitrary positive integers.  In particular,  we have $N_1^{d_\FS}(k,n) \le \frac{1}{k! (k+1)!} n^{2k} + O(n^{2k-1})$,  which improves the leading coefficient of the bound in \eqref{eq:NdFS}.  

Finally, Theorem~\ref{thm:equiangularAD} gives,  for any angle distance and positive integers $2 \le k \le n$,
\[
N_1^\delta(k,n) \le 
\begin{cases}
\dbinom{\binom{n+1}{2} + 1}{2} - \dbinom{n+1}{2} \quad &\text{if~} k = 2,  \\[1.5em]
\dbinom{\binom{n+1}{2} + k - 1}{k}  - \dbinom{n}{2} \dbinom{\binom{n}{2} -1}{k-3} - n\dbinom{\binom{n}{2}}{k-3} \quad &\text{if~} 3 \le k \le n. 
\end{cases}
\]
Consequently,  this improves \eqref{eq:Ndelta} by subtracting a lower order term.  Although \eqref{eq:Ndelta} was used in \cite[Corollary~2.3]{Balla} to show $N_1^{\delta}(k,n) = \Theta(n^{2k})$,  our improvement implies that \eqref{eq:Ndelta} is in fact unattainable.

\subsection*{Organization}
Section~\ref{sec:prelim} briefly reviews the geometry of the Grassmannian.  In Section~\ref{sec:poly},  we refine the dimension counting technique underlying the polynomial method.  Section~\ref{sec:s-distance} is devoted to $s$-distance subspaces for general $s$: we prove the upper bounds for $N_s^{d_\C}(k,n)$ and $N_s^{d_\FS}(k,n)$,  respectively. In Section~\ref{sec:equiangular},  we focus on equiangular subspaces,  and obtain an improved upper bound for $N_1^{\delta}(k,n)$.

\section{Preliminaries}\label{sec:prelim}
\subsection{Hilbert Function and Hilbert Series}
Let $X$ be a subvariety of $\mathbb{P}^{N-1}$ over a field $\mathbb{K}$ and let $\mathbb{I}_X \subseteq R \coloneqq \mathbb{K}[x_1,\dots,  x_N]$ be its ideal.  The \emph{Hilbert function} \cite[Definition~5.1.1]{greuel2008singular} of $X$ is defined as 
\[
H_{X} (d) \coloneqq \dim_\mathbb{K} (R/ \mathbb{I_X})_d,
\]
where $S_d$ is the degree-$d$ piece of a graded algebra $S$.  The \emph{Hilbert series} \cite[Definition~5.1.1]{greuel2008singular} of $X$ is the generating function:
\[
\mathrm{Hilb}_{X}(t) := \sum_{d=0}^\infty H_{X}(d) t^d.
\]
Moreover,  it is well-known \cite[Corollary~5.1.5]{greuel2008singular} that there exists a univariate polynomial $h_X(t) \in \mathbb{Q}[t]$,  called the \emph{Hilbert polynomial},  such that $h_X(d) = H_X(d)$ for all sufficiently large $d$. 
\subsection{Grassmannians as projective varieties}
Let $k \le n$ be positive integers.  We denote by $\Gr(k,n)$ the set of all $k$-dimensional subspaces in $\mathbb{R}^n$.  It is well-known that $ \Gr(k,n)$ is a projective variety via the Pl\"{u}cker embedding:
\begin{equation}\label{eq:projGr}
\iota: \Gr(k,n) \to \mathbb{P}^{\binom{n}{k}-1},\quad \iota(\mathbb{U}) \coloneqq [u_1\wedge \cdots \wedge u_k] = [\det (U_I)]_{I \subseteq [n],\; |I| = k}.
\end{equation}
where $u_1,\dots,  u_k$ is a basis of $\mathbb{U}$ and $U_I$ is the $k \times k$ submatrix of $U = [u_1,\dots,  u_k] \in \mathbb{R}^{n \times k}$ formed by rows index by $I$.
\begin{lemma}\cite{Braun}\label{thm:Hseries}
The Hilbert series of $\iota(\Gr(k, n))$ is given by:
\[
\mathrm{Hilb}_{\iota(\Gr(k, n))}(t) = \frac{N_k(n - k + 1; t)}{(1 - t^k)^{k(n - k) + 1}},
\]
where
\[
N_k(r; t) \coloneqq  \sum_{j=0}^{(r-1)(k-1)} c_k(r, j) t^j,\quad c_k(r, j) \coloneqq \sum_{l=0}^{j} (-1)^{j - l} \binom{kr + 1}{j - l}
\prod_{i=0}^{k-1} \frac{\binom{r + i + l}{r}}{\binom{r + i}{r}}.
\]
\end{lemma}

\subsection{Grassmannians as affine varieties}
We recall \cite{lim2025degree,devriendt2025two} that $\Gr(k,n)$ can be embedded into $\mathsf{S}^2(\mathbb{R}^n)$ as an affine variety via the map 
\begin{equation}\label{eq:affine}
\varepsilon: \Gr(k,n) \to \mathsf{S}^2(\mathbb{R}^n),\quad \varepsilon(\mathbb{V}) = P_{\mathbb{V}}.
\end{equation}
Here $P_{\mathbb{V}}$ is the orthogonal projection matrix of $\mathbb{V}$.  If the column vectors of $V \in \mathbb{R}^{n\times k}$ form an orthonormal basis of $\mathbb{V}$,  then $P_{\mathbb{V}} = VV^\tp$.  Moreover,  the defining ideal $\mathbb{I}_{k,n}$ of $\varepsilon (\Gr(k,n))$ is generated by elements of $\tr(X) - k$ and $X^2 - X$,  where $X = (x_{i,j})_{i,j\in [n]}$ and $x_{i,j} = x_{j,i}$ is a variable for each $(i,j)\in [n] \times [n]$. The following is a direct consequence of \cite[Theorem~4.3]{lim2025degree}. 
\begin{lemma}[Hilbert polynomial of Grassmannian]\label{lem:hp}
Let $k,  n$ be positive integers such that $k\le n/2$.  The  Hilbert polynomial of $\varepsilon(\Gr(k,n)) \subseteq \mathsf{S}^2 (\mathbb{R}^n)$ is 
\[
h_{\varepsilon(\Gr(k,n))} (d) \coloneqq  \dim \left( \frac{\mathbb{R}[X]_{\le d}}{\mathbb{R}[X]_{\le d} \cap \mathbb{I}_{k,n}} \right) = \frac{\alpha_{k,n} \sum_{\lambda \succeq \delta_k} 
A_{\lambda,k} B_{\lambda,k} C_{\lambda,k}}{(k(n-k))!} d^{k(n-k)} + O(d^{k(n-k)-1}),
\]
where $\delta_k \coloneqq (k-1, \dots, 1, 0)$,  $C_{\lambda,k}$ is the coefficient of the Jack symmetric function 
$J_\lambda^{(2)}(x)$ in the expansion
\[
\prod_{1 \le i < j \le k} (x_i + x_j) = \sum_{\lambda \succeq \delta_k} C_{\lambda,k} J_\lambda^{(2)}(x),
\]
and 
\begin{align*}
A_{\lambda,k} &\coloneqq \prod_{i=1}^k \Gamma\Bigl( n-2k + 1 + \lambda_i + \frac{k-i}{2} \Bigr), 
\quad 
B_{\lambda,k} \coloneqq \prod_{1\le i < j \le k} 
\frac{\Gamma\bigl(\lambda_i - \lambda_j + \frac{j-i+1}{2}\bigr)} 
{\Gamma\bigl(\lambda_i - \lambda_j + \frac{j-i}{2}\bigr)},  \\
\alpha_{k,n} &\coloneqq 
\begin{cases}
\dfrac{2^{k(n-k-1)}}{\displaystyle \prod_{\substack{1 \le i \le k \\ i < j \le n/2}} (j-i)(n-j-i)}
& \text{if $n$ is even and $k \le n/2 - 1$},\\[1em]
\dfrac{2^{k(n-k)}}{\displaystyle \prod_{\substack{1 \le i \le k \\ i < j \le (n-1)/2}} (j-i)(n-i-j) \prod_{i=1}^{k} (n-2i)}
& \text{if $n$ is odd},\\[1em]
\dfrac{2^{k(k-1)+1}}{\displaystyle \prod_{1 \le i < j \le k} (j-i)(2k - j - i)}
& \text{if $n = 2k$}.
\end{cases}
\end{align*}
\end{lemma}
\section{A refinement of the dimension counting}\label{sec:poly}
The \emph{polynomial method} is a powerful technique in combinatorics, which has found wide applications in extremal problems involving sets, subspaces, and distances. The central idea is to associate algebraic data—such as polynomials—to a finite configuration and to use dimension counting to deduce combinatorial bounds.  The goal of this section is to establish a refinement of the dimension counting in the polynomial method for problems involving algebraic varieties.  

We denote by $\mathbb{R}[x_1,\dots,  y_n]_d$ the space of degree $d$ homogeneous  polynomials with variables $x_1,\dots,  x_n$,  and by $\mathbb{R}[x_1,\dots,  x_n; y_1,\dots,  y_n]_{d,e}$ the space of bihomogeneous polynomials with variables $x_1,\dots,  x_n,y_1,\dots,  y_n$,  whose degree in $x$'s and $y$'s is $k$ and $l$,  respectively.  Moreover,  we set
\begin{align*}
\mathbb{R}[x_1,\dots,  x_n]_{\le d} &\coloneqq \bigoplus_{j=0}^d \mathbb{R}[x_1,\dots,  x_n]_{j},  \\
\mathbb{R}[x_1,\dots,  x_n; y_1,\dots,  y_n]_{\le d, \le e} \coloneqq 
&\bigoplus_{i,j = 1}^{d,e} \mathbb{R}[x_1,\dots,  x_n; y_1,\dots,  y_n]_{i,  j}.
\end{align*}
By definition,  the \emph{Hilbert function of a projective variety} $X \subseteq \mathbb{P}^{n-1}$ is 
\[
H_X(d) \coloneqq \binom{n + d - 1}{d} - \dim \left( \mathbb{R}[x_1,\dots,  x_n]_{d} \cap \mathbb{I}_X \right).
\]
\begin{proposition}[Refined dimension counting]\label{prop:rank}
Suppose $d_1,\dots,  d_k,e_1,\dots, e_k$ are non-negative integers and $p \in \oplus_{s=1}^k \mathbb{R}[x_1,\dots,  x_n; y_1,\dots,  y_n]_{d_s,e_s}$.  Given vectors $v_1,\dots,  v_m \in X$,  we define $f_i(x) \coloneqq p(x,  v_i) \in \oplus_{s=1}^k \mathbb{R}[x_1,\dots,  x_n]_{d_s}$ for each $i\in [m]$ and $M \coloneqq (p(v_i,v_j))_{i,j\in [m]} \in \mathbb{R}^{m \times m}$.  We have the following: 
\begin{enumerate}[label = (\alph*)]
\item\label{prop:rank:eq1} If $X \subseteq \mathbb{R}^n$ is an affine variety containing $S$,  then
\[
\rank (M) \le \dim \mathbb{V} \le  \sum_{s=1}^k \binom{n + d_s - 1}{d_s} - \dim \left( \mathbb{I}_X \cap  \bigoplus_{s=1}^k \mathbb{R}[x_1,\dots,  x_n]_{d_s} \right)
\]
\item \label{prop:rank:eq2} If $X \subseteq \mathbb{P}^{n-1}$ is a projective variety containing $[v_1],\dots,  [v_m]$,  then
\[
\rank \left( (p(v_i,  v_j))_{i,j\in [m]} \right)  \le  \dim \mathbb{V} \le  \sum_{s=1}^k H_X(d_s).
\] 
\end{enumerate}
Here $\mathbb{V} \coloneqq \spa_{\mathbb{R}} \lbrace f_1|_{X},\dots,  f_m|_{X}
\rbrace$. 
\end{proposition}
\begin{proof}
We observe that $f_1|_X,\dots,  f_m|_X$ are elements in 
\begin{equation}\label{prop:rank:iso}
\frac{ \mathbb{I}_X + \bigoplus_{s=1}^k \mathbb{R}[x_1,\dots,  x_n]_{d_s}}{\mathbb{I}_X} \simeq \frac{\bigoplus_{s=1}^k \mathbb{R}[x_1,\dots,  x_n]_{d_s}}{\mathbb{I}_X \cap \bigoplus_{s=1}^k \mathbb{R}[x_1,\dots,  x_n]_{d_s}}.
\end{equation}
Since the $i$-th row of $M$ is the vector consisting of evaluations of $f_i$ at $a_1,\dots,  a_m$ for each $i \in [m]$,  we have $\rank (M) \le \dim \mathbb{V}$ and \ref{prop:rank:eq1} follows immediately.  If $X$ is projective,  its ideal $\mathbb{I}_X$ is homogeneous.  In this case,  \eqref{prop:rank:iso} becomes 
\[
\frac{ \mathbb{I}_X + \bigoplus_{s=1}^k \mathbb{R}[x_1,\dots,  x_n]_{d_s}}{\mathbb{I}_X} \simeq \frac{\bigoplus_{s=1}^k \mathbb{R}[x_1,\dots,  x_n]_{d_s}}{\mathbb{I}_X \cap \bigoplus_{s=1}^k \mathbb{R}[x_1,\dots,  x_n]_{d_s}} \simeq \bigoplus_{s=1}^k (\mathbb{R}[x_1,\dots,  x_n]/\mathbb{I}_X)_{d_s},
\]
and this completes the proof of \ref{prop:rank:eq2}.
\end{proof}
\begin{remark}
If $\{(d_1,e_1), \dots, (d_k,e_k) \} = [d] \times [d]$ for some positive integer $d$,  then the inequality in Proposition~\ref{prop:rank}--\ref{prop:rank:eq1} can be written as 
\begin{equation}\label{prop:rank:eq3}
\rank (M) \le \dim \mathbb{V} \le H_X(d),
\end{equation}
where $H_X: \mathbb{N} \to \mathbb{N}$ is the Hilbert function of the affine variety defined by 
\[
H_X(d) = \dim \left( 
\frac{\mathbb{I}_X + \mathbb{R}[x_1,\dots,  x_n]_{\le d}}{\mathbb{I}_X}  \right)
\]  
If either $X = \mathbb{R}^n$ or $X = \mathbb{P}^n$,  then \ref{prop:rank:eq1}--\ref{prop:rank:eq2} reduce to inequalities that are extensively used in the polynomial method \cite{babai1992linear}.  
\end{remark}

\section{$s$-distance subspaces}\label{sec:s-distance}
Let $\delta: \Gr(k,n) \times \Gr(k,n) \to \mathbb{R}$ be a function.  We recall that an \emph{$s$-distance set} of $\Gr(k,n)$ with respect to $\delta$ consists of $\mathbb{U}_1,\dots,  \mathbb{U}_N \in \Gr(k,n)$ such that 
\[
\left\lvert \{ \delta(\mathbb{U}_i, \mathbb{U}_j): i,j \in [N],\; i \ne j \} \right\rvert \le s.
\]
The subspaces $\mathbb{U}_1,\dots,  \mathbb{U}_N$ are correspondingly called \emph{$s$-distance subspaces}. We denote by $N^{\delta}_s(k,n)$ the maximum cardinality of an $s$-distance set in $\Gr(k,n)$ with respect to $\delta$.  In this section,  we apply Proposition~\ref{prop:rank} to obtain upper bounds on $N^{\delta}_s(k,n)$ when $\delta$ is either the Chordal or Fubini-Study distance.  As a consequence,  we generalize and improve the upper bounds \eqref{eq:NChordal} and \eqref{eq:NdFS}.
\subsection{$s$-distance subspaces for Chordal distance}
The Chordal distance on $\Gr(k,n)$ is defined as 
\begin{equation}\label{eq:Chordal}
d_\C(\mathbb{U},  \mathbb{V}) := \left( \sum_{i=1}^k \sin^2 \theta_i \right)^{1/2}
  = \left( k - \operatorname{tr}( U U^\tp VV^\tp) \right)^{1/2},
\end{equation}
where $U, V$ are $n \times k$ matrices whose column vectors form an orthonormal basis of $\mathbb{U}$ and $\mathbb{V}$,  respectively.  
\begin{theorem}[$s$-distance subspaces for $d_\C$]\label{thm:s-dC}
Given positive integers $k \le n$,  we have 
\[
N^{d_\C}_s(k,n) \le \frac{\alpha_{k,n} \sum_{\lambda \succeq \delta_k} 
A_{\lambda,k} B_{\lambda,k} C_{\lambda,k}}{(k(n-k))!} s^{k(n-k)} + O(s^{k(n-k)-1})
\]
for sufficiently large $s$.  Here $\alpha_{k,n}$,  $\delta_k$,  $A_{\lambda,k}$,  $B_{\lambda,k}$ and $C_{\lambda,k}$ are numbers defined as in Lemma~\ref{lem:hp}, 
\end{theorem}
\begin{proof}
Let $S$ be an $s$-distance set in $\Gr(k,n)$ with respect to $d_\C$.  Suppose
\[
\left\lbrace
d_\C(\mathbb{U},\mathbb{V}): \mathbb{U} \ne \mathbb{V} ,\; \mathbb{U},  \mathbb{V} \in S
\right\rbrace = \{a_1,\dots,  a_s\}.
\]
Let $X$ (resp.  $Y$) be the $n\times n$ symmetric matrix whose $(i,j)$-th element is a variable $x_{i,j}$ (resp.  $y_{i,j}$) where $(i,j) \in [n] \times [n]$.  We consider the polynomial \[
p(X,Y) \coloneqq \prod_{t=1}^s \left(
\tr(XY) + a_t^2 - k) \right)
 \in \mathbb{R}[X;Y] = \mathbb{R}[x_{1,1},\dots,  x_{n,n}; y_{1,1},\dots,  y_{n,n}] 
\]
It is clear that $p \in \mathbb{R}[X;Y]_{\le s, \le s}$.  Denote $P_i \coloneqq \varepsilon(\mathbb{U}_i) \in \mathsf{S}^2(\mathbb{R}^n)$ where $\varepsilon$ is the map defined by \eqref{eq:affine}.  Then we have $p(P_i,  P_j) = \delta_{i,j} (a_1\cdots a_s)^2$ for any $i,j \in [n]$.  According to Proposition~\ref{prop:rank} and  \eqref{prop:rank:eq3},  we have $m \le H_{\varepsilon(\Gr(k,n))}(s)$.  The proof is completed by Lemma~\ref{lem:hp} and the fact that $ H_{\varepsilon(\Gr(k,n))}(s) = h_{\varepsilon(\Gr(k,n))}(s)$ when $s$ is sufficiently large \cite[Corollary~5.1.5]{greuel2008singular}.  
\end{proof}
\begin{remark}\label{rmk:dc}
For simplicity,  we denote $d(k,n) \coloneqq \alpha_{k,n} \sum_{\lambda \succeq \delta_k} 
A_{\lambda,k} B_{\lambda,k} C_{\lambda,k}$.  By Lemma~\ref{lem:hp},  the value of $d(k,n)$ can be easily obtained.  For example,  we have 
\[
d(1,n) = 2^{n-1},\quad d(2,n) = 2 \binom{2n-4}{n-2},\quad d(3,n) = \frac{(8n - 25)(2n - 9)!!}{(n-2)!} 2^{2n-6}.
\] 
We refer the interested reader to \cite{lim2025degree} for more explicit formulas of $d(k,n)$.  
\end{remark}
\subsection{$s$-distance subspaces for Fubini-Study distance}
The \emph{Fubini-Study distance} on $\Gr(k,n)$ is defined as
\begin{equation}\label{eq:FS}
d_{\FS}(\mathbb{U},  \mathbb{V}) := \arccos \left| \det(U^\top V) \right| = \arccos \left( \prod_{i=1}^k \cos \theta_i \right),
\end{equation}
where $U, V$ are $n \times k$ matrices whose column vectors form an orthonormal basis of $\mathbb{U}$ and $\mathbb{V}$,  respectively.  In this subsection,  we will obtain an upper bound on $N^{d_{\FS}}_s(k,n)$.  To begin with,  we establish an estimate for the Hilbert function $H_{\iota(\Gr(k,n)}$ of $\Gr(k,n)$ under the Pl\"{u}cker embedding $\iota: \Gr(k,n) \to \mathbb{P}^{\binom{n}{k} - 1}$ defined in \eqref{eq:projGr}.
\begin{lemma}\label{lem:sumHilb}
Fix integers $k,s \ge 1$.  We have
\[
\sum_{i=1}^s H_{\iota(\Gr)(k,n)}(2i)
= \left[ \prod_{j=1}^{2s} \frac{(j-1)!}{(j+k-1)!} \right] n^{2ks} + O\left(n^{2ks-1}\right).
\]
\end{lemma}
\begin{proof}
According to Lemma~\ref{thm:Hseries},  we have
\[
H_{\iota(\Gr)(k,n)}(m)
=\sum_{j =0,\;
k \mid (m - j)  }^{(n-k)(k-1)}
c_k(n-k+1,j) \binom{k(n-k)+\frac{m-j}{k}}{\frac{m-j}{k}}.
\]
For fixed $i$ and $m=2i$, the dominant contribution arises from the term $j=2i$.  Thus,  we obtain 
\[
H_{\iota(\Gr)(k,n)}(2i) \sim c_k(n - k + 1,2i) \sim \left[\prod_{j=1}^{2i} \frac{(j-1)!}{(j+k-1)!}\right] n^{2ki}.
\]
Summing over $i \in [s]$, the largest power of $n$ comes from $i=s$, yielding
\[
\sum_{i=1}^s H_{\iota(\Gr)(k,n)}(2i)
= \left[\prod_{j=1}^{2s} \frac{(j-1)!}{(j+k-1)!}\right] n^{2ks} + O(n^{2ks-1}).  \qedhere
\]
\end{proof}

The theorem that follows provides an upper bound on $N^{d_{\FS}}_s(k,n)$. 
\begin{theorem}[$s$-distance subspaces for $d_{\FS}$]\label{thm:s-distFS}
For any positive integers $k \le n$,  we have 
\[
N_s^{d_\FS}(k,n) \le \left[\prod_{j=1}^{2s} \frac{(j-1)!}{(j+k-1)!}\right] n^{2ks}  + O\left(n^{2ks-1}\right).
\]
\end{theorem}
\begin{proof}
Let \( S  \subseteq \mathrm{Gr}(k,n) \) be an $s$-subset with respect to the Fubini–Study distance.  Assume that 
\[
\left\lbrace
d_\FS(\mathbb{U},\mathbb{V}): \mathbb{U} \ne \mathbb{V} ,\; \mathbb{U},  \mathbb{V} \in S
\right\rbrace = \{a_1,\dots,  a_s\} \subseteq (0,  \pi/2 ].
\]
For each \( \mathbb{U} \in S \),  we define the function on $\Gr(k,n)$:
\[
f_ \mathbb{U}( \mathbb{X} ) := \prod_{i=1}^{s} \left( \det(U^\tp X)^2 - \cos^2(a_i) \right)
\]
where $U, X$ are $n \times k$ matrices whose column vectors form an orthonormal basis of $\mathbb{U}$ and $\mathbb{X}$.  By construction, \( f_{\mathbb{U}}(\mathbb{V}) = 0 \) for all \( \mathbb{V} \ne \mathbb{U} \in S \).  Moreover,  we have \( f_{\mathbb{U}}(\mathbb{U}) \ne 0 \) since \( \det(U^\tp U)^2 = 1\).  Therefore, the set \( \{ f_{\mathbb{U}}: \mathbb{U} \in S \} \) is linearly independent.  

The Cauchy-Binet formula implies 
\[
\det(U^\tp X) = \sum_{I \subseteq [n], \; |I| = k} \det(U_I) \det(X_I),
\]
where $Y_I$ denotes the submatrix of $Y \in \mathbb{R}^{n \times k}$ obtained by rows indexed by $I \subseteq [n]$.  Therefore,  for a fixed $U$,  $ \det(U^\tp X)^2 $ can be written as a quadratic homogeneous polynomial in the Pl\"{u}cker coordinates of \( X \).  Consequently,  \( f_{\mathbb{U}} \) is a linear combination of homogeneous polynomials of degrees \( 2, \dots,  2s \) in the Pl\"{u}cker coordinates.  In other words,  we have 
\[ 
f_{\mathbb{U}} \in \bigoplus_{i=1}^{s} \mathbb{R}[\iota(\Gr(k,n))]_{2i},
\]
Since \( \{ f_{\mathbb{U}}: \mathbb{U} \in S \} \) is linearly independent,  we may conclude from Lemma~\ref{lem:sumHilb} that
\[
|S| \leq \sum_{i=1}^{s} H_{\Gr(k,n)}(2i) = \left[\prod_{j=1}^{2s} \frac{(j-1)!}{(j+k-1)!}\right] n^{2ks} + O\left(n^{2ks-1}\right). \qedhere
\]
\end{proof}
\begin{remark}
When $s = 1$,  it was shown in \cite{Balla} that 
\[
N_1^{d_{\FS}} (k,n) \le \binom{\binom{n}{k} + 1}{2} = \frac{1}{2(k!)^2} n^{2k} + O(n^{2k-1}).
\]
Theorem~\ref{thm:s-distFS} improves the leading coefficient to $1/ (k! (k+1)!)$.
\end{remark}

\subsection{$s$-distance lines}
Let $\mathbb{S}^{n-1}$ be the unit sphere in $\mathbb{R}^n$.  A \emph{spherical $s$-distance} set is a subset $S \subseteq \mathbb{S}^{n-1}$ such that 
\[
|\{ \langle u,  v\rangle \in \mathbb{R}: u,v \in S,\; u \ne v\}| \le s.
\]
In the literature \cite{Rankin55,DGS77,BDKS18},  spherical $s$-distance sets are also called spherical $L$-codes 

We denote by $g(n,s)$ the maximum cardinality of spherical $s$-distance sets in $\mathbb{S}^{n-1}$.  We recall that the existing general upper bound on $M(n,s)$ (cf.  \cite{DGS77,  GY18}) is 
\[
g(n,s) \le \binom{n + s - 1}{n-1} + \binom{n + s -2}{n-1}.
\]

Given two lines $\mathbb{U},  \mathbb{V} \in \Gr(1,n) = \mathbb{P}^{n-1}$,  we have $d_{\C} (\mathbb{U},  \mathbb{V}) = |\sin \theta|$ and $d_{\FS}(\mathbb{U},  \mathbb{V}) = \theta$,  where $\theta \in [0,\pi/2]$ is the angle between $\mathbb{U}$ and $\mathbb{V}$.  Since $\mathbb{P}^{n-1}$ is obtained from the sphere $\mathbb{S}^{n-1}$ by identifying its antipodes,  each $\mathbb{U} \in \mathbb{P}^{n-1}$ corresponds to two points $u,  -u\in \mathbb{S}^{n-1}$.  Let $\pi: \mathbb{S}^{n-1} \to \mathbb{P}^{n-1}$ be the projection map sending $\pm u$ to $\mathbb{U}$.  

Thus,  for any $s$-distance set $S$ in $\mathbb{P}^{n-1}$ with respect to $\delta \in \{d_{\C},  d_{\FS}\}$,  there is some spherical $2s$-distance subset $\widetilde{S} \subseteq \mathbb{S}^{n-1}$ such that $\pi (\widetilde{S}) = S$ and $|\widetilde{S}| = |S|$.  This implies 
\[
|S|  = |\widetilde{S}| \le g(n,2s),
\]
from which we obtain the following upper bound on $L_{s}^{\delta}(1,n)$,  whose leading term coincides with those in Theorems~\ref{thm:s-dC} and \ref{thm:s-distFS}.
\begin{proposition}\label{cor:spherical}
Suppose $n$ and $s$ are positive integers.  For $\delta \in \{d_{\C},  d_{\FS}\}$,  we have  
\[
N_{s}^{\delta}(1,n) \le \binom{n + 2s - 1}{n-1} + \binom{n + 2s -2}{n-1}.
\]
\end{proposition}
\begin{remark}
It was shown in \cite{CHS96} that $N_{1}^{d_{\C}}(k,n) \le \binom{n+1}{2}$.  For $k = 1$,  Proposition~\ref{cor:spherical} generalizes this upper bound to arbitrary $s$.
\end{remark}

\section{Equiangular Subspaces}\label{sec:equiangular}
Given two $k$-dimensional subspaces $\mathbb{U},  \mathbb{V}$ in $\mathbb{R}^n$,  the $i$-th \emph{principal angle} between $\mathbb{U}$ and $\mathbb{V}$ is $\theta_i(\mathbb{U},\mathbb{V}) = \arccos (\sigma_i) \in [0,\pi/2]$ for each $i \in [k]$.  Here $\sigma_1 \ge \cdots \ge \sigma_k$ are singular values of $U^\tp V \in \mathbb{R}^{k \times k}$,  and $U,  V$ are $n \times k$ orthonormal matrices whose columns form an orthogonal basis of $\mathbb{U}$ and $\mathbb{V}$,  respectively.  An \emph{angle distance} on $\Gr(k,n)$ is a function $\delta : \Gr(k,n) \times \mathrm{Gr}(k,n) \to \mathbb{R}_{\ge 0}$ such that \( \delta(\mathbb{U}, \mathbb{V}) \in \{\theta_1(\mathbb{U}, \mathbb{V}), \dots, \theta_k(\mathbb{U}, \mathbb{V})\} \) for all \( \mathbb{U},  \mathbb{V} \in \Gr(k,n) \).  We recall that elements in a $1$-distance subset of $\Gr(k,n)$ are said to be \emph{equiangular} \cite{Blokhuis93,IFPB17,Balla}.  In this section,  we focus on the maximum number $N_1^{\delta}(k,n)$ of equiangular subspaces when $\delta$ is an angle distance. We will derive an upper bound on $N_1^{\delta}(k,n)$,  which improves \eqref{eq:Ndelta}. 

Let $\varepsilon: \Gr(k,n) \to \mathsf{S}^2(\mathbb{R}^n)$ be the embedding of $\Gr(k,n)$ into $\mathsf{S}^2(\mathbb{R}^n)$ as a real affine variety defined in \eqref{eq:affine}.  The ideal of $\varepsilon(\Gr(k,n))$ is generated by $\tr(X) - k$ and $X^2 - X$.  For simplicity,  we write 
\[
\mathbb{R}[X] \coloneqq \mathbb{R}[x_{i,j}: 1 \le i \le j \le n],\quad 
\mathbb{I}_{k,n}  \coloneqq \langle \tr(X) - k,  X^2 - X \rangle.
\]
\begin{lemma}\label{lem:pd}
For any positive integer $d,k,n$ such that $2 \le k \le n$,  we have 
\[
\dim \left( \mathbb{R}[X]_d/ \left( \mathbb{I}_{k,n} \cap \mathbb{R}[X]_{d} \right) \right) \le \binom{ M_{n+1}+ d - 1}{d} - p_d(n),
\] 
where $M_s \coloneqq \binom{s}{2}$ for each integer $2 \le s$ and
\[
p_d(n) = \begin{cases}
M_{n+1} \quad &\text{if~} d =2, \\[5pt]
M_n \binom{M_n-1}{d-3} + n \binom{M_n}{d-3} \quad &\text{if~} 3 \le d \le M_n + 2.
\end{cases}
\]
\end{lemma}
\begin{proof}
We construct linearly independent polynomials in $\mathbb{I}_{k,n} \cap \mathbb{R}[X]_{d}$.  Define the index set 
\[
\mathcal{I} \coloneqq \{(i,j) : 1 \le i \le j  \le n \},\quad \mathcal{I}_0 \coloneqq \{(i,j) : 1 \le i < j  \le n \}.
\]
Given $(i,j) \in \mathcal{I}$,  we let $E_{i,j}$ be the $n \times n$ symmetric matrix whose elements are all zero except for the $(i,j)$-th and $(j,i)$-th ones,  which are equal to one.  
\begin{enumerate}[label = (\alph*)]
\item We consider quadratic polynomials
\[
f_{i,j}(X) \coloneqq (x_{1,1} + \cdots + x_{n,n}) x_{i,j} - k \left( \sum_{l=1}^n x_{i,l} x_{j,l} \right), \quad (i,j) \in \mathcal{I}.
\]
Notice that $f_{i,j} (X) =  \left( \operatorname{tr}(X) - k \right) x_{i,j}  - k (X^2 - X)_{i,j} \in \mathbb{I}_{k,n} \cap \mathbb{R}[X]_2$.  We claim that $\{f_{i,j}: (i,j) \in \mathcal{I}\}$ is a linearly independent set. Suppose $f \coloneqq \sum_{(i,j) \in \mathcal{I}} c_{i,j} f_{i,j} = 0$ in $\mathbb{R}[X]$ for some $c_{i,j} \in \mathbb{R}$.  We want to prove that $c_{i,j} = 0$ for all $(i,j) \in \mathcal{I}$.  For any $1 \le r < s  \le n$ and any $t \in [n] \setminus \{r,s\}$,  we have 
\begin{align*}
f(E_{r,r} - E_{s,s}) &= -k (c_{r,r} + c_{s,s}) = 0,  \\ 
f(E_{r,r} + E_{s,s} - 2 E_{t,t}) &= -k(c_{r,r} + c_{s,s} + 2 c_{t,t}) = 0,   \\ 
f(E_{r,s} + E_{r,r} - E_{t,t}) &= - k \left( 2 c_{r,r} + c_{s,s} + c_{t,t} + c_{r,s} \right) = 0.
\end{align*}
This implies $c_{i,j} = 0$ for $(i,j) \in \mathcal{I}$.
\item When $d \ge 3$,  we consider for each $(i,j) \in \mathcal{I}$ polynomials
\[
g_{i,j}(X) \coloneqq \tr(E_{i,j} X)\in \mathbb{R}[X]_1,  \quad h_{i,j}(X) \coloneqq x_{i,j} f_{i,j}(X)\in \mathbb{I}_{k,n} \cap \mathbb{R}[X]_3.
\]
We denote
\begin{align*}
\Lambda_d  &\coloneqq \Bigl\{ (\alpha_1,\dots,  \alpha_{d-2})\in \mathcal{I} \times \mathcal{I}_0^{d-3}: \alpha_s \ne \alpha_t,  \;  s \ne t \in [d-2] \Bigr\},  \\
\mathcal{P}_d  &\coloneqq \left\lbrace P_I   \in \mathbb{R}[X]_d: P_I = h_{\alpha_1}g_{\alpha_2}\cdots g_{\alpha_{d-2}},  \; I = (\alpha_1,\dots,  \alpha_{d-2}) \in \Lambda_d
\right\rbrace.
\end{align*}
It is clear that $\mathcal{P}_d \subseteq \mathbb{I}(k,n) \cap \mathbb{R}[x]_d$ and $|\mathcal{P}_d| = p_d(n)$.  

We claim that polynomials in $\mathcal{P}_d$ are linearly independent.  To prove the claim,  we suppose $\sum_{J \in \Lambda_d} c_{J} P_J= 0$ for some $c_{J} \in \mathbb{R}$.  For each $I = (\alpha_1,\dots,  \alpha_{d-2}) \in \Lambda_d$,  we consider the $n \times n$ symmetric matrix 
\[
B_I \coloneqq (1-\delta_{\alpha_1}) E_{1,1} +  \sum_{\alpha \in I} E_{\alpha}. 
\]
Here $\delta_{\alpha_1}$ is the Kronecker delta function and the notation ``$\alpha \in I$" means $\alpha$ is a component of $I$.  In the rest of the proof,  we adopt this abused notation.  For $I = (\alpha_1,\dots,\alpha_{d-2}) \in \Lambda_d$,  $\beta \in \mathcal{I}_0$ and $\gamma \in \mathcal{I}$,
\[
g_\beta (B_I) = \begin{cases}
1 \quad &\text{if~} \beta \in I\\
0 \quad &\text{otherwise}
\end{cases},  \quad 
(B_I)_{\gamma} = \begin{cases}
1 \quad &\text{if~}\delta_{\alpha_1} = 1\text{~and~} \gamma \in I\\
0 \quad &\text{if~}\delta_{\alpha_1} = 1\text{~and~} \gamma \not\in I\\
1 \quad &\text{if~}\delta_{\alpha_1} = 0 \text{~and~} \gamma \in I\cup \{(1,1)\}\\
0 \quad &\text{if~}\delta_{\alpha_1} = 0 \text{~and~} \gamma \not\in I\cup \{(1,1)\}
\end{cases}.
\]
Moreover,  we have $h_{\gamma} (B_I) = (B_I)_{\gamma} f_{\gamma} (B_I)  = ( (B_I)_{\gamma} - k  (B_I^2)_{\gamma}) (B_I)_{\gamma}$.  Since $(B_I)_{\gamma} = 0$ or $1$, $h_{\gamma} (B_I) = 0$ or $h_{\gamma} (B_I) \ne 0$.  Notice that $P_J(B_I) = h_{\beta_1}(B_1) g_{\beta_2}(B_I) \cdots g_{\beta_{d-2}}(B_I)$ for each $J = (\beta_1,\dots,  \beta_{d-2}) \in \Lambda_d$.  When $\delta_{\alpha_1} = 1$,  we obtain
\[
P_J(B_I)  \ne 0 \iff \alpha_1 = \beta_1 \text{~and~} \{\alpha_2,\dots,  \alpha_{d-2}\} =  \{\beta_2,\dots,  \beta_{d-2}\}. 
\]
This implies $c_J = 0$ for any $J = (\beta_1,\dots,  \beta_{d-2}) \in \Lambda$ with $\delta_{\beta_1} = 1$. If $\delta_{\alpha_1} = \delta_{\beta_1} = 0$,  then we may derive 
\[
P_J(B_I)  \ne 0 \iff  \{\alpha_1,  \alpha_2,\dots,  \alpha_{d-2}\} =  \{\beta_1,  \beta_2,\dots,  \beta_{d-2}\},
\]
which leads to $c_J = 0$ for $J = (\beta_1,\dots,  \beta_{d-2}) \in \Lambda$ with $\delta_{\beta_1} = 0$.  \qedhere
\end{enumerate}
\end{proof}
As a consequence of Lemma~\ref{lem:pd},  we obtain an upper bound on $N_1^\delta(k,n)$ when $\delta$ is an angle distance.
\begin{theorem}[Equiangular subspaces for angle distances]\label{thm:equiangularAD}
Let $k,n$ be positive integers such that $2 \le k \le n$.  For any angle distance $\delta$ on $\Gr(k,n)$,  we have 
\[
N_1^\delta(k,n) \le 
\begin{cases}
\binom{M_{n+1} + 1}{2} - M_{n+1} \quad &\text{if~} k = 2,  \\[5pt]
\binom{M_{n+1} + k - 1}{k}  - M_n \binom{M_n -1}{k-3} - n\binom{M_n}{k-3} \quad &\text{if~} 3 \le k \le n. 
\end{cases}
\]
Here $M_s \coloneqq \binom{s}{2}$ for each integer $2 \le s$.
\end{theorem}
\begin{proof}
Let $\{\mathbb{U}_1, \dots,  \mathbb{U}_m\} \subseteq \Gr(k,n)$ be an equiangular subset with  respect to $(\delta,\alpha)$ for some $\alpha \in (0,\pi/2]$.  For each $i \in [m]$,  we consider 
\[
f_i(X) \coloneqq \det\left( U_i^\tp X U_i -  \frac{\cos (\alpha) \tr(X)}{k} I_k\right) \in \mathbb{R}[X]_{k},\quad P_i \coloneqq U_i U_i^\tp \in \varepsilon (\Gr(k,n)) \subseteq \mathsf{S}^2(\mathbb{R}^n),
\]
where column vectors of $U_i \in \mathbb{R}^{n \times k}$ form an orthonormal basis of $\mathbb{U}_i$.  Since $f_i(P_j)  = \delta_{i,j} (1 - \cos(\alpha))^k$,  $\{f_1,\dots,  f_m\}$ is linearly independent in $\mathbb{R}[X]_k / \left( \mathbb{I}_{k,n} \cap \mathbb{R}[X]_k \right)$.  This implies 
\[
N_1^\delta(k,n) \le \dim \mathbb{R}[X]_k / \left( \mathbb{I}_{k,n} \cap \mathbb{R}[X]_k \right).
\]
The desired upper bound follows immediately from Lemma~\ref{lem:pd}.
\end{proof}
\begin{remark}
It was proved in \cite[Theorem~2.1]{Balla} that
\begin{equation}\label{eq:balla}
N_1^{\delta}(k,n) \le  \binom{ M_{n+1} +k-1}{k}.
\end{equation}
The upper bound in Theorem~\ref{thm:equiangularAD} improves \eqref{eq:balla} by subtracting a correction term that accounts for the additional algebraic constraints.  
\end{remark}

\bibliographystyle{plain}
\bibliography{bound}
\end{document}